\documentclass[11pt]{amsart}

\usepackage{color}
\usepackage{amsfonts}
\usepackage{amssymb}

\usepackage{amsmath}
\usepackage{hyperref}
\usepackage{enumerate}
\usepackage{graphicx}

\newcommand{\cii}[1]{_{ {}_{ #1}}}

\newcommand{\eq}[2][label]{\begin{equation}\label{#1}#2\end{equation}}

\newcommand{\av}[2]{\langle #1\rangle\cii {#2}}
\newcommand{\df}{\buildrel\rm{def}\over=}

\newcommand{\Bel}{\mathbf B}
\newcommand{\bell}{\mathbf b}

\newcommand{\half}{\tfrac12}

\newcommand{\pd}{\partial}

\newcommand{\al}{\alpha}
\newcommand{\eps}{\varepsilon}
\newcommand{\la}{\lambda}

\newcommand{\cD}{\mathcal D}

\newcommand{\cT}{\mathcal T}

\newcommand{\bE}{\mathbb E}

\newcommand{\bR}{\mathbb R}

\newcommand{\bfU}{\mathbf U}

\newtheorem{theorem}{Theorem}[section]
\newtheorem{cor}[theorem]{Corollary}
\newtheorem{lemma}[theorem]{Lemma}
\newtheorem{prop}[theorem]{Proposition}
\theoremstyle{definition}

\newtheorem{defin}[theorem]{Definition}

\numberwithin{equation}{section}

\newcounter{vremennyj}

\def\Xint#1{\mathchoice
{\XXint\displaystyle\textstyle{#1}}%
{\XXint\textstyle\scriptstyle{#1}}%
{\XXint\scriptstyle\scriptscriptstyle{#1}}%
{\XXint\scriptscriptstyle\scriptscriptstyle{#1}}%
\!\int}
\def\XXint#1#2#3{{\setbox0=\hbox{$#1{#2#3}{\int}$ }
\vcenter{\hbox{$#2#3$ }}\kern-.6\wd0}}

\def\dashint{\Xint-}

\begin{document}
\title[Bellman function for square function in $L^p$]{Obstacle problems generated by the estimates of square function}
\author[I.~Holmes, A.~Volberg]{I.~Holmes,  A.~Volberg}
\thanks{I. Holmes is supported by National Science Foundation as an NSF Postdoc under Award No.1606270, A.~Volberg is partially supported by the NSF DMS-1600065.  This paper is  based upon work supported by the National Science Foundation under Grant No. DMS-1440140 A. Volberg was in residence at the Mathematical Sciences Research Institute in Berkeley, California, during the Spring and Fall  2017 semester. }

\address{Department of Mathematics, Michigan State University, East Lansing, MI 48823, USA}
\email{holmesir@math.msu.edu \textrm{(I.\ Holmes)}}
\address{Department of Mathematics, Michigan State University, East Lansing, MI 48823, USA}
\email{volberg@math.msu.edu \textrm{(A.\ Volberg)}}

\makeatletter
\@namedef{subjclassname@2010}{
  \textup{2010} Mathematics Subject Classification}
\makeatother
\subjclass[2010]{42B20, 42B35, 47A30}
\keywords{}

\begin{abstract} 
In this note we give the formula  for the Bellman function associated with the problem considered by B. Davis in \cite{Davis} in 1976. In this article the estimates of the type $\|Sf\|_p \le C_p \|f\|_p$, $p\ge 2$, were considered for the dyadic square function operator $S$, and Davis found the sharp values of  constants $C_p$. However, along with the sharp constants one can consider a more subtle characteristic of the above estimate. This quantity is called the Bellman function of the problem, and it seems to us that it was never proved that the confluent hypergeometric function from Davis' paper (second page) basically gives this Bellman function. Here we fill out this gap by finding the exact Bellman function of the unweighted $L^p$ estimate for operator $S$. We  cast the proofs in the language of obstacle problems.  For the sake of comparison, we also find the Bellman function of weak $(1,1)$ estimate of $S$. This formula was suggested by Bollobas \cite{Bollobas} and proved by Osekowski \cite{Os2009}, so it is not new, but we like to emphasize the common approach to those two Bellman functions descriptions.

\end{abstract}
\maketitle

\section[Obstacle problems for  square function operator]{Obstacle problems for unweighted square function operator: Burkholder--Gundy--Davis function}
\label{BGBB}

Recall that $h\cii J$ denotes the normalized in $L^2$ Haar function supported on interval $J$.
Let now $g$ be a test function on an interval $I$, then 
$$
g= \av gI {\bf 1}\cii I +\sum_{J\in \cD(I)}\! \Delta\cii J g
$$
with $\Delta\cii J g = (g, h\cii J)h\cii J$.
The square function of $g$ is the following aggregate:
$$
Sg(x)\df \Big(\!\!\sum_{\genfrac{}{}{0pt}{}{J\in \cD(I)}{x\in J}}\! (\Delta\cii J g)^2(x)\,\Big)^{1/2}\,.
$$

Marcinkiewicz--Paley inequalities \cite{Mar} relate the norms of $g-\av gI$ and $Sg$, claiming that for certain situations these norms can be equivalent.

Let $W(t)$ be the standard Brownian motion starting at zero, and $T$ be any  stopping time. Below $\|f\|_\alpha$ stands for $L^\alpha$ norm.

D. Burkholder \cite{BuMart}  P. Millar \cite{Mi}, 
A. A. Novikov \cite{No}, D. Burkholder and R. Gundy   \cite{BuGa},  B.~Davis  \cite{Davis},  found the following  norm estimates 
\eq[DavisBr01]{
c_\alpha  \| T^{1/2}\|_{\alpha}\leq \|W(T)\|_{\alpha}, \quad  1<\alpha<\infty, \,\, \| T^{1/2}\|_{\alpha}<\infty; 
}
\eq[DavisBr02]{
 \| W(T)\|_{\alpha}\leq C_\alpha\|T^{1/2}\|_{\alpha}, \quad 0<\alpha<\infty.  
 }

Davis ~\cite{Davis}   found the best possible values of constants above.

It was explained in \cite{Davis} that the same sharp estimates  (\ref{Davis01}) and (\ref{Davis02}) below  hold with $W(T)$ replaced by an integrable function $g$ on $[0,1]$, and $T^{1/2}$ replaced by  the dyadic square function of $g$.

More precisely, Davis proved that
\begin{align}
&c_\alpha  \| Sg\|_{\alpha}\leq \|g\|_{\alpha}, \quad  2\le\alpha <\infty; \label{Davis01}\\
& \| g\|_{\alpha}\leq C_\alpha\|Sg\|_{\alpha}, \quad 0<\alpha\le 2.  \label{Davis02}
\end{align}
with the same constants as above, and these constants are sharp in those ranges of $\alpha$ and $\beta$. Inequality \eqref{Davis02} with the same sharp constant as in \eqref{DavisBr02} but for the range $\beta \ge 3$ was proved by G.~Wang \cite{Wang}. In the range $\beta\in (2,3)$ the sharp constant in \eqref{Davis02} is not known to the best of our knowledge. The same can be said about \eqref{Davis01} in the range $\alpha\in (1, 2)$. Notice also that Wang's results  are proved for  square functions of conditionally symmetric martingales. So Wang's setting is more general than the dyadic setting presented here.

\bigskip

Our reasoning here first follows  the original proof by B. Davis of estimates \eqref{DavisBr01}, \eqref{Davis01}. based on the construction of a corresponding Bellman function. Davis considers two problems: 1) the continuous one, where stopping time serves as the replacement of the square function operator, 2) and a discrete one, concerning the dyadic square function operator  $S$ itself. 

For the continuous problem he  defines the Bellman function (on page  699 of \cite{Davis} it is called $v(t,x)$). But he seems to be leaving the finding of the Bellman function for the estimate of $S$ outside of the scope of his paper. 

We just fill out this small gap in the present note. This is done by Theorem \ref{bfU6_th}, the main part is Section \ref{the_smallest}.

\bigskip

But first we wish to cast the proofs in the language of obstacle problems. To prepare the ground we start with explanation what  are obstacle problems related to square function estimates.

\subsection{Obstacle problems related to square function estimates}
\label{OSqF}

We will always work with functions on some interval $I$, and $\cT\df \cT(I)$ is the class of test functions. We say that $f\in \cT$ if $f$ is constant
on each dyadic interval from $\cD_N(I)$ for some finite $N$. 

The main players will be  an ``arbitrary" function $O\colon\bR\times \bR_+\to \bR$ (an obstacle) and a  function $U\colon\bR\times \bR_+\to \bR$, $U\ge O$, satisfying the following inequality
\begin{equation}
\label{MISq}
2U(p, q) \ge U(p+a, \sqrt{a^2+q^2}) + U(p-a, \sqrt{a^2 +q^2})\,.
\end{equation}
We will call this {\it  the main inequality}, functions U satisfying the main inequality 
will be precisely  Bellman functions of various estimates concerning square function operator.  

Of course the existence of $U$ majorizing $O$ and satisfying \eqref{MISq} is not at all ensured. 

Notice that \eqref{MISq} is invariant under taking infimum.

\begin{defin} 
We call the smallest $U$ satisfying the main inequality and majorizing $O$  {\it the heat envelope} of $O$.
\end{defin}

We would like to find the heat envelope of some specific $O$. 
\begin{theorem}
\label{misq1}
Let $U$ satisfy main inequality \eqref{MISq}. Then for any $f\in \cT(I)$
\begin{equation}
\label{USq}
\av{U(f, \sqrt{q^2+(Sf)^2})}I \le U(\av fI, q)\,.
\end{equation}
\end{theorem}

Here is a corollary relating the main inequality with square function estimates.

\begin{cor}
\label{misq1cor}
Let $U$ satisfy main inequality \eqref{MISq}. Then for any $f\in \cT(I)$
\begin{equation}
\label{USq}
\av{U(f, Sf)}I \le U(\av fI, 0)\,.
\end{equation}
\end{cor}

Before proving Theorem \ref{misq1},  we wish to answer the question, when, given $O$, one can find a finite valued function majorizing $O$ and satisfying the main inequality.

\begin{theorem}
\label{fatU}
Let
\eq[fatUeq]
{
\bfU(p, q) \df  \sup_{\genfrac{}{}{0pt}{}{f\in \cT(I)}{\av fI=p}} \av{O(f, \sqrt{q^2+S^2f}) }I\,.
}
If this function is finite valued, then it satisfies the main inequality.
\end{theorem}

Now we wish to formulate results that can be considered as converse to Theorem \ref{misq1}.
They concern the obstacle problem for \eqref{MISq}.

As was already mentioned, by this we understand finding $U$ satisfying \eqref{MISq} and majorizing a
 given function (obstacle) $O\colon \bR\times \bR_+\to \bR$. It turns out that one can 
 give ``simple" conditions necessary and sufficient for the solvability of the obstacle problem.

\begin{theorem}
\label{misq2}
Let an obstacle function $O$, and  a function $F\colon\bR\to \bR$ satisfying 
$F(p)\ge O(p, 0)$ be given.  A finite valued function $U$   satisfying
\begin{itemize}
\item main inequality \eqref{MISq}
\item $U\ge O$   
\item  $U(p, 0)\le F(p)$ 
\end{itemize} 
exists if and only if
\begin{equation}
\label{OSq}
\av{O(f, Sf)}I \le F(\av fI)\,,\quad \forall  f\in \cT\,.
\end{equation}
\end{theorem}

It will be especially important to use this result with one special $F$: $F=0$.
\begin{theorem}
\label{misq3}
Given an obstacle function $O$\textup, to find $U$  satisfying main inequality \eqref{MISq} and such that $U\ge O$  and  $U(p, 0)\le 0$\textup, it is necessary and sufficient to have
\begin{equation}
\av{O(f, Sf)}I \le 0,\qquad \forall  f\in \cT\,.
\end{equation}
\end{theorem}

\begin{proof}[Proof of theorem \ref{misq1}]
Below by $\bE_k$ we denote  the expectation with  respect to $\sigma$-algebra generated by dyadic intervals of family $\cD_k$.
We first prove Theorem \ref{misq1}. Let $f\in \cT$, and let $N$ be such that $f$ is constant on each $J\in \cD_N(I)$.
Let us consider two siblings $\ell_+, \ell_-\in \cD_N(I)$ with the same father $\ell\in \cD_{N-1}(I)$. 

Denote  $p\df\av f\ell$ and let $\av f{\ell_+}= p+a)$,  then $\av f{\ell_-}= p-a$, and $f(x) =p\pm a$
for all $x\in \ell_\pm$ correspondingly.  Notice that  for all $x\in \ell$, $|\Delta\cii \ell f(x)|=|a|$,
and put $q_1\df   \sqrt{S^2f(x)-a^2}$, where $x\in \ell_\pm$ {\bf(}the value  $Sf(x)$  is the same for all $x\in\ell${\bf)}.
By the main inequality we have
\begin{align*}
&\int_{\ell_+} \!U(f(x), \sqrt{q^2+S^2f(x)}) \,dx + \int_{\ell_-} \!U(f(x), \sqrt{q^2+S^2f(x)}) \,dx =
\\
&|\ell|\left(\half U(p+a, \sqrt{q^2+a^2 +q_1^2}) + \half U(p+a, \sqrt{q^2+a^2 +q_1^2})\right)\le 
\\
&|\ell| U (p, q) = \int_\ell U(f_1(x), \sqrt{q^2+S^2f_1(x)})\, dx,
\end{align*}
where $f_1\df  \bE_{N-1} f$. We can continue now  by recursion. We denote $f_k\df  \bE_{N-k} f$, $ k=1 \dots N$. So $f_N(x)= \bE_0 f= \av fI{\bf 1}\cii I$. Notice that
$S f_N=0$ identically, and after repeating the above recursion  $N+1$ times we come to
\begin{equation}
\label{USq2}
\int_I U(f(x), \sqrt{q^2+S^2f(x)})\, dx \le |I| U(\av fI, q),
\end{equation}
which is the claim of the theorem.
\end{proof}

\begin{proof}[Proof of Theorem \ref{fatU}]
It is clear by its definition and by rescaling, that $\bfU$ does not depend on the interval $I$, where test functions are defined. 
Therefore, given the data $(p+a, \sqrt{a^2+q^2})$, we can find a function $f_+$ optimizing \hbox{$\bfU(p+a, \sqrt{a^2+q^2})$}
up to $\eps$, and we can think as well that it lives on $I_+$. Similarly, given the data $(p-a, \sqrt{a^2+q^2})$, 
we can find a function $f_-$ optimizing $\bfU(p-a, \sqrt{a^2+q^2})$ up to $\eps$, and we can think as well that it lives on $I_-$.

Concatenate functions $f_\pm$ on $I_\pm$ to the following function:
$$
f(x)= \begin{cases}  f_+(x), \,\,x\in I_+\\
f_-(x), \, \,x\in I_-
\end{cases}
$$
Since $\av fI=p$, we have
\begin{align*}
&\bfU(p, q) \ge  \av{O(f, \sqrt{q^2+S^2f}) }I 
\\
&=\half  \av{O(f, \sqrt{q^2+S^2f})  }{I_+} +  
\half  \av{O(f, \sqrt{q^2+S^2f}) }{I_-}
\\
&=\half  \av{O(f_+, \sqrt{q^2+a^2+S^2f_+) }}{I_+} + \half  \av{O(f_-, \sqrt{q^2+a^2+S^2f_+}) }{I_-}
\\
&\ge \half  \bfU(p+a, \sqrt{a^2+q^2})-\eps +  \half  \bfU(p+a, \sqrt{a^2+q^2})-\eps\,.
\end{align*}
As $\eps$ is an arbitrary positive number we are done.
\end{proof}

Now we prove Theorem \ref{misq2}.

\begin{proof}
First we prove the ``if" part. We are given an obstacle $O$ and a function $F$ such that $F(p)\ge O(p, 0)$. We defined
$$
\bfU(p, q) =  \sup_{\genfrac{}{}{0pt}{}{f\in \cT(I)}{\av fI=p}} \av{O(f, \sqrt{q^2+S^2f}) }I\,.
$$
It is obvious that $\bfU(p, q)\ge O(p, q)$, one just plugs the constant function $f= p{\bf 1}\cii I$.

It is also clear that $\bfU(p, 0)\le F(p)$. Indeed,
$$
\bfU(p, 0) =\sup_{\genfrac{}{}{0pt}{}{f\in \cT(I)}{\av fI=p}} \av{O(f, Sf)}I \le F(\av fI) = F(p)
$$
by  assumption \eqref{OSq}.   Hence $\bfU(p, 0)$ is finite valued. 

The fact that function $\bfU$ defined as above satisfies the main inequality ~\eqref{MISq} follows from Theorem \ref{fatU}.
 Then by \eqref{MISq} it is finite valued.

Now we prove the ``only if " part. We need to prove that
$$
\av{O(f, Sf)}I \le F(\av{f}I)
$$
if there exits a  majorant $U$ of $O$ satisfying the main inequality and satisfying $U(p, 0) \le F(p)$.
This is easy:
$$
\av{O(f, Sf)}I \le \av{U(f, Sf)}I\le U(\av fI, 0)\le  F(\av fI),
$$
where the second inequality follows from Corollary \ref{misq1cor} we have
\end{proof}

\medskip

The following theorem sums up the results of this section.
\begin{theorem}
\label{minUmaxU}
There exists a finite valued function $U$ majorizing $O$ and satisfying the main inequality if and only if $\bfU$ from \eqref{fatUeq}  is finite valued.
Moreover, if $\bfU$ is finite valued,
then the  infimum of functions $U$ majorizing $O$ and satisfying the main inequality is equal to $\bfU$ .
\end{theorem}
\begin{proof}
We already saw in Theorem \ref{fatU} that $\bfU$ from \eqref{fatUeq} (if finite valued) is one of those functions $U$ that majorize $O$ and satisfy the main inequality. 

On the other hand, for any function $U$ that majorize $O$ and satisfy the main inequality we know from Theorem \ref{misq1} that for any test function $f$ and any non-negative $q$ the following holds
$$
U(\av fI, q) \ge \av{U(f, \sqrt{q^2+S^2f})}I \ge \av{O(f, \sqrt{q^2+S^2f})}I\,. 
$$
Take now the supremum over test functions in the right hand side. By definition we obtain $\bfU(\av fI, q)$. Theorem is proved.
\end{proof}

We will consider the following examples.

\medskip

\noindent{\bf Example 0.} Davis function that gives  the proof of \eqref{Davis01} for $\alpha\ge 2$. 
Here the obstacle function will be
\begin{equation}
\label{Dobst}
O_0(p, q)=c_\alpha^\alpha |q|^\alpha- |p|^\alpha,
\end{equation}
where the best value of $c_\alpha$ was found by Davis \cite{Davis}.

\noindent{\bf Example 1.} Bollob\'as function. Here the obstacle function will be
\begin{equation}
\label{Bobst1}
O_1(p, q) = {\bf 1}_{q\ge 1} - C |p|,
\end{equation}
where the best value of $C$ was suggested by B.~Bollob\'as~\cite{Bollobas}. This was verified  
by A.~Os\c ekowski~\cite{Os2009}, see also~\cite{HoIvV}.

\noindent{\bf Example 2.} Bollob\'as function. Here the obstacle function will be
\begin{equation}
\label{Bobst2}
O_2(p, q) = {\bf 1}_{p^2+q^2\ge 1} - C |p|,
\end{equation}
where the best value of $C$ was suggested by B.~Bollobas~\cite{Bollobas} and also verified 
by A.~Os\c ekowski~\cite{Os2009}, see also~\cite{HoIvV}.

\noindent{\bf Example 3.}  Bellman function associated with  the Chang--Wilson--Wolff theorem.
\begin{equation}
\label{ChWWobst}
O_3(p, q;\la) = {\bf 1}_{[\la, \infty)}(p){\bf 1}_{[0, 1]}(q)\,.
\end{equation}

Function $U$ is not fully known in the case. It is ``almost" found in \cite{NaVaVo}.

\subsection{Davis obstacle problem}
\label{Dobst_sec}

In this section we want to find the minimal value $c_\alpha$ for which there exists a function 
$\bfU\colon \mathbb{R}^2\to \mathbb{R}$ that solves the problem with the obstacle function of Example~0, i.\,e.,
\eq[bfU6]
{
\bfU(p, q)\df\sup\{\av{c_\alpha^\alpha\big(q^2+(Sf)^2\big)^{\alpha/2}-|f|^\alpha}{I}\colon\av fI=p\}\,.
}
In other words, we want to find the heat envelope of $O_0$.
Let $\alpha\geq2$ and let  $\beta=\frac{\alpha}{\alpha-1}\leq2$  be the conjugate exponent of $\alpha$. 
Let  
\begin{align}\label{series}
N_{\alpha }(x)&\df{}_1F_1\left(-\frac\alpha2,\frac12,\frac{x^2}2\right)\nonumber 
\\
&=\sum_{m=0}^{\infty}\frac{(-2x^{2})^{m}}{(2m)!}\frac{\alpha}2\left(\frac{\alpha}2-1\right)\cdots
\left(\frac{\alpha}2-m+1\right)
\\
&=1-\frac\alpha2x^2+\frac\alpha{12}\left(\frac\alpha2-1\right)x^4\ldots\nonumber 
\end{align}
be the confluent hypergeometric function.  $N_{\alpha }(x)$ satisfies the Hermite differential equation 
\begin{align}\label{hermit}
N''_{\alpha }(x)-xN'_{\alpha }(x)+\alpha N_{\alpha}(x)=0 \quad \text{for} \quad  x\in \mathbb{R}
\end{align}
with initial conditions $N_{\alpha}(0)=1$ and $N'_{\alpha}(0)=0$. Let $c_\alpha$ \label{dex_u_alpha} be the smallest 
positive zero of $N_{\alpha}$. 

Set 
\begin{equation}
\label{dex_u}
u_{\alpha}(x) \df  
\begin{cases}
-\dfrac{\alpha c_\alpha^{\alpha-1}}{N'_{\alpha}(c_\alpha)} N_{\alpha}(x), & 0\leq |x|\leq c_\alpha;\\[10pt]
\ \ c_\alpha^{\alpha}-|x|^{\alpha}, & c_\alpha \leq |x|.
\end{cases}
\end{equation}
Clearly $u_{\alpha}(x)$ is $C^{1}(\mathbb{R}) \cap C^2(\mathbb{R}\setminus{\{c_\alpha\}})$ 
smooth even concave function.  The concavity follows from Lemma~\ref{root} on the 
page~\pageref{root} and the fact that  $N'_{\alpha}(c_\alpha) <0$. Finally we define 
\eq[upq]
{
U(p,q) \df 
\begin{cases} 
|q|^{\alpha} u_{\alpha}\left( \frac{|p|}{|q|}\right), &\quad q\neq 0,
\\
\quad-|p|^{\alpha}, &\quad q=0.
\end{cases}
}
In this section we are going to prove the following result.
\begin{theorem}
\label{bfU6_th}
Function $\bf U$ from \eqref{bfU6} is equal to $U$ written above in \eqref{upq}.
\end{theorem}

For the first time the function $U(p,q)$ appeared in~\cite{Davis}. Later it was also used 
in~\cite{Wang, Wang2} in the form  $\widetilde{u}(p,t)=U(p,\sqrt{t})$, $t\geq 0$. 
Since want to prove that
$$
U=\bfU\,,
$$
at first we will verify  the following properties:
\eq[obstacle6]{
U(p,q) \geq |q|^{\alpha} c_\alpha^{\alpha} - |p|^{\alpha}, \quad (p,q) \in\mathbb{R}^2, 
}
\eq[neravenstvo]{
2U(p,q) \geq U(p+a,\sqrt{a^{2}+q^{2}})+U(p-a, \sqrt{a^2+q^2}), \,(p,q,a) \in\mathbb{R}^3.
}

When these two properties get proved, Theorem \ref{minUmaxU} ensures that
\eq[bfUU]
{
\bfU \le U\,.
}
This inequality is the most difficult part of Theorem ~\ref{bfU6_th}.

We called~\eqref{obstacle6} the {\em obstacle condition}, and \eqref{neravenstvo} 
the {\em main inequality}. 
The infinitesimal form  of~\eqref{neravenstvo} is
\eq[infini]
{
\frac1qU_{q}+U_{pp}\leq0\,,
}
which follows from the main inequality by expanding it into Taylor's series with respect to $a$ near  the origin and comparing the second order terms.

\bigskip

First we  check \eqref{infini}.  On  domain $p/ q\in (-c_\alpha,c_\alpha)$, $q>0,$ this follows 
from~\eqref{upq} and the first line of~\eqref{dex_u}. Moreover, on this domain we have equality 
$U_q/q+U_{pp}=0$, which easily follows from~\eqref{hermit}. 
On the complementary domain, where $|p|\ge c_\alpha q$, we have
\begin{align*}
\frac1q U_q+U_{pp}&=\alpha (c_\alpha^\alpha q^{\alpha-2}- (\alpha -1) |p|^{\alpha-2})
\\
&=\alpha q^{\alpha-2}c_\alpha^{\alpha-2}\Big(c_\alpha^2-(\alpha -1)\big(\frac{|p|}{c_\alpha q}\big)^{\alpha-2}\Big)<0,
\end{align*}
because $\alpha\ge 2$ and, as we will see below in Lemma~\ref{root}, $c_\alpha \le 1$. In fact, we need more, 
we need also to check that in the sense of distributions~\eqref{infini} is also satisfied, 
but this calculation we leave for the reader.

Inequality~\eqref{infini} guarantees that 
$$
X_{t} = U(W(t),\sqrt{t}) \quad \text{is a supermartingale for} \quad t \geq 0.
$$  
In fact, using It\^o's formula, we get
$$
dX(t)= \frac1{2\sqrt{t}}\frac{\pd U}{\pd q}dt+\frac12\frac{\pd^2 U}{\pd p^2}dt+\frac{\pd U}{\pd p}dW(t),
$$
and therefore~\eqref{infini} implies that $dX(t)-\frac{\pd U}{\pd p}dW(t)\le0$, so $X(t)$ is 
a supermartingale.

Finally, the supermartingale property gives us the second inequality below
$$
\mathbb{E}(T^{\frac{\alpha}{2}}c_\alpha^{\alpha} - 
|B_{T}|^{\alpha}) \stackrel{\eqref{obstacle6}}{\leq} \mathbb{E}U(B_{T}, \sqrt{T}) \leq U(0,0)=0,
$$
which yields~\eqref{Davis01}. 

Now we are going to prove  that $U(p,q)$ is the minimal function with properties~\eqref{obstacle6} 
and~\eqref{neravenstvo}. 

\bigskip

The next step is to go from infinitesimal version \eqref{infini} to finite difference 
inequality~\eqref{neravenstvo}. For that we need several lemmas.

\begin{lemma}\label{root}
The minimal positive root $c_\alpha$ of $N_\alpha$ has the following properties.

{\rm 1)} The estimate $0<c_\alpha\leq 1$ is valid for  $\alpha \geq 2$.

{\rm 2)} $c_\alpha$ is decreasing in $\alpha>0$.

{\rm 3)}  $N'_{\alpha}(t)\leq 0, \,\, N''_{\alpha}(t) \leq 0$ on $[0, c_\alpha]$ for $\alpha>0$. 
\end{lemma}

\begin{proof}
Consider $G_{\alpha}(t) \df e^{-t^{2}/4}N_{\alpha}(t)$. 
Notice that the zeros of $G_{\alpha}$ and $N_{\alpha}$ are the same.  It follows from (\ref{hermit}) that 
\begin{align}\label{hyper1}
G''_{\alpha} + \left(\alpha+\frac{1}{2}-\frac{t^{2}}{4} \right) G_{\alpha}=0, \quad G_{\alpha}(0)=1 \quad \text{and} \quad G'_{\alpha}(0)=0.
\end{align}
Besides we know that the solution is even. Consider the critical case $\alpha=2$. 
In this case $G_{2}(t)=e^{-t^{2}/4}(1-t^{2})$ and the smallest positive zero is $s_{2} = 1$. Therefore it follows from the 
Sturm comparison principle that $0<c_\alpha<1$ for $\alpha>2$ (see below). Moreover, the 
same principle applied to $G_{\alpha_{1}}$ and $G_{\alpha_{2}}$  with $\alpha_{1}> \alpha_{2}$ implies that $G_{\alpha_{1}}$ 
has a zero inside the interval $(-s_{\alpha_{2}}, s_{\alpha_{2}})$. Thus  we conclude that  $c_\alpha$ is decreasing in $\alpha$.  

To verify that $N'_{\alpha}, N''_{\alpha} \leq 0$ on $[0,c_\alpha]$, first we claim that
\begin{align*}
N_{\alpha_{2}} \geq N_{\alpha_{1}}  \quad \text{on} \quad [0, s_{\alpha_{1}}]
\end{align*}
for $\alpha_{1}>\alpha_{2}>0$. Indeed the proof works in the same way as 
the proof of Sturm's comparison principle. For the convenience of the reader we decided to include the argument. 
As before, consider $G_{\alpha_{j}} = e^{-t^{2}/4} N_{\alpha_{j}}$. It is enough to show that $G_{\alpha_{2}} \geq G_{\alpha_{1}}$ 
on $[0,s_{\alpha_{1}}]$. It follows from (\ref{hyper1}) that $G''_{\alpha_{2}}(0) > G''_{\alpha_{1}}(0)$. 
Therefore,  using the Taylor series expansion at the point $0$, we see that the claim is true at some 
neighborhood of zero, say $[0, \varepsilon)$ with $\varepsilon$ sufficiently small.
 Next we assume the contrary, i.e.,  that there is a point $a \in [\varepsilon, s_{\alpha_{1}}]$ 
 such that $G_{\alpha_{2}}\geq G_{\alpha_{1}}$ on $[0,a]$,  $G_{\alpha_{2}}(a)=G_{\alpha_{1}}(a)$ and $G'_{\alpha_{2}}(a)<G'_{\alpha_{1}}(a)$ 
 (notice that the case $G'_{\alpha_{2}}(a)=G'_{\alpha_{1}}(a)$,  
 by the uniqueness theorem for ordinary differential equations, 
 would imply that $G_{\alpha_{2}}=G_{\alpha_{1}}$ everywhere, which is impossible). Consider the Wronskian 
\begin{align*}
W=G_{\alpha_{1}}'G_{\alpha_{2}} - G_{\alpha_{1}} G'_{\alpha_{2}}.
\end{align*}
We have  $W(0) = 0$ and $W(a)=G_{\alpha_{1}}(a) (G'_{\alpha_{1}}(a) - G'_{\alpha_{2}}(a)) \geq  0$. On the other hand, we have 
\begin{align*}
W' = (\alpha_{2} -\alpha_{1}) G_{\alpha_{1}}G_{\alpha_{2}} < 0 \quad \text{on} \quad [0,a),
\end{align*}
which is a clear contradiction, and this proves the claim. 

It follows from (\ref{series}) that 
\begin{align}
\label{recurent}
N''_{\alpha} = -\alpha N_{\alpha-2},
\end{align}
 and inequalities  $N_{\alpha-2} \geq  N_{\alpha}\geq 0$ on $[0,c_\alpha]$ imply that 
 
$$
N''_{\alpha}\leq 0 \,\,\text{on}\,\, [0, c_\alpha]\,. 
$$ 
Since $N'_{\alpha}(0)=0$, and $N''_{\alpha}\leq 0$ on $[0, c_\alpha]$, 
we must have $N'_{\alpha}\leq 0$ on $[0,c_\alpha]$. 
\end{proof}

\bigskip

\begin{lemma}
\label{blema}
For any $p \in \mathbb{R}$\textup, the function 
\begin{align}\label{convsq}
t \mapsto U(p, \sqrt{t}) \quad \text{is convex for} \quad t\geq 0\,.
\end{align}
\end{lemma}

\begin{proof}
Without loss of generality, assume that $p\geq 0$. We recall that 
$U(p,\sqrt{t}) = t^{\alpha/2}u_{\alpha}(p/\sqrt{t})$.  Since $\alpha \geq 2$, the only 
interesting case to consider is when $p/\sqrt{t} <c_\alpha$ (otherwise $ t^{\alpha/2}$ is convex). 
In this case we have $U(p,\sqrt{t}) = \kappa_\alpha t^{\alpha/2} N_{\alpha}(p/\sqrt{t})$, 
where $\kappa_\alpha$ is a positive constant. In particular, by~\eqref{hermit} we have  
\hbox{$U(p,\sqrt{t})_{t}+\half U(p,\sqrt{t})_{pp}=0$}. 
Using~\eqref{recurent}, we obtain
\begin{align*}
U(p,\sqrt{t})_{t} = -\frac{U(p,\sqrt{t})_{pp}}{2} = -\frac{\kappa_\alpha}{2}t^{\frac{\alpha}{2}-1}N''_{\alpha}(p/\sqrt{t}) =
 \frac{\alpha\kappa_\alpha}{2}t^{\frac{\alpha-2}{2}}N_{\alpha-2}(p/\sqrt{t}).
\end{align*}
Therefore, it would be enough to show that for any $\gamma \geq 0$,  the function   $x^{-\gamma}N_{\gamma}(x)$
is decreasing  for $x \in (0,s_{\gamma+2})$. 
Differentiating, and using~\eqref{hermit} again, we obtain 
\begin{align*}
\frac{\mathrm{d}}{\mathrm{d}x} \, \left(\frac{N_{\gamma}(x)}{x^{\gamma}}\right) =\frac{N''_{\gamma}(x)}{x^{\gamma+1}},
\end{align*}
which  is nonpositive by Lemma~\ref{root}. 
\end{proof}

The next lemma, together with Lemma~\ref{blema} and~\eqref{infini}, implies that $U(p,q)$ 
satisfies~\eqref{neravenstvo}. 

\begin{lemma}[Barthe--Mauery~\cite{BM}]
\label{barko}
Let $J$ be a convex subset of $\mathbb{R}$, and let $V(p,q) \colon J \times \mathbb{R}_{+} \to \mathbb{R}$ be such that 
\begin{align}
&V_{pp}+\frac{V_{q}}{q} \leq 0 \quad \text{for all}\quad (p,q) \in J\times \mathbb{R}_{+}; \label{inffed}\\
&t \mapsto V(p, \sqrt{t}) \quad \text{is convex for each fixed} \quad p \in J.\label{amoz}
\end{align}
Then for all  $(p,q,a)$ with  $p\pm a \in J$ and $q\geq 0$, we have 
\begin{align}
2V(p,q)\geq V(p+a, \sqrt{a^{2}+q^{2}})+V(p-a,\sqrt{a^{2}+q^{2}}). \label{toloba}
\end{align}
\end{lemma}
The lemma says that the global finite difference inequality (\ref{toloba}) is in fact implied by its infinitesimal form (\ref{inffed}) under the extra condition (\ref{amoz}).
\begin{proof}
The argument is borrowed from \cite{BM}. 

Without loss of generality assume $a\geq 0$. Consider the process 
\begin{align*}
X_{t} = V(p+W(t), \sqrt{q^{2}+ t}),  \quad t\geq 0.
\end{align*}
Here $W(t)$ is the standard Brownian motion starting at zero. It follows from It\^o's formula  together with (\ref{inffed}) that $X_{t}$ is a supermartingale. 
Indeed, by It\^o's formula we have 
\begin{align*}
X_{t} = X_{0} + \int_{0}^{t}V_{p} \,dW(t) + \frac12 \int_{0}^{t}\left(V_{pp}+\frac{V_{q}}{\sqrt{q^{2}+t}}\right) \,dt
\end{align*}
and notice that the drift term is negative. 
Let $\tau$ be the stopping time such that $W(\tau)$ hits $a$ or $-a$, i.\,e.

\begin{align*}
\tau = \inf\{ t \geq 0 \colon W(t) \notin (-a,a)\}.
\end{align*} 

The supermartingale property of  $X_t$  and concavity \eqref{amoz} yield the following 
chain of inequalities:
\begin{align*}
V(p,q) &= X_{0} \geq \mathbb{E} X_{\tau}=\mathbb{E} V(p+W(\tau), \sqrt{q^{2}+\tau})  
\\
&=P(W(\tau)=-a) \mathbb{E}(V(p-a, \sqrt{q^{2}+\tau}) | W(\tau)=-a)
\\
&\qquad+P(W(\tau)=a) \mathbb{E}(V(p+a, \sqrt{q^{2}+\tau}) | W(\tau)=a)
\\
&=\half\Big(\mathbb{E}(V(p-a, \sqrt{q^{2}+\tau}) | W(\tau)=-a)
\\
&\qquad+\mathbb{E}(V(p+a, \sqrt{q^{2}+\tau}) | W(\tau)=a) \Big)
\\
& \geq \half\Big(V\big( p-a, \sqrt{q^{2}+ \mathbb{E}(\tau|W(\tau)=-a) }\big)
\\
&\qquad+V\big( p+a, \sqrt{q^{2}+ \mathbb{E}(\tau|W(\tau)=a) }\big) \Big)
\\
&=\half\Big( V\big(p-a, \sqrt{q^{2}+a^{2}}\big) +  V\big(p+a, \sqrt{q^{2}+a^{2}}\big) \Big)\,.
\end{align*}

Notice that  we have used  $P(W(\tau)=a)=P(W(\tau)=-a)=1/2$, $\mathbb{E}(\tau | W(\tau)=a) =\mathbb{E}(\tau | W(\tau)=-a)=a^{2}$, and the fact that the map $t \mapsto V(p, \sqrt{t})$ is convex together with Jensen's inequality.  
\end{proof}

\bigskip

\subsection{Majorization of the obstacle function.}

We have finished the proof of inequality \eqref{neravenstvo}. Now we are going to check \eqref{obstacle6} from page \pageref{obstacle6}. Let 
$$
\kappa_\alpha= -\frac{\alpha c_\alpha^{\alpha-1}}{N_\alpha'(c_\alpha)}.
$$

Function $\kappa_\alpha N_\alpha$ in the first line of \eqref{dex_u} 
is equal to function $g\df c_\alpha^\alpha  -  x^\alpha$ at $x=c_\alpha$. 
To prove that $\kappa_\alpha N_\alpha\ge g$ on 
$[0, c_\alpha]$, thus, it is enough to prove 
$\kappa_\alpha N_\alpha' \le g'$ on this interval. 
At point $c_\alpha$ these derivatives coincide by the choice of 
$\kappa_\alpha$. Notice that $\kappa_\alpha>0$ and that $N_\alpha'$ and $g'$ are negative. Therefore, to check that $-\kappa_\alpha N_\alpha' \ge -g'$  it is enough to show that function $-N_\alpha'/x^{\alpha -1}$ is decreasing  on $[0, c_\alpha]$, i.\,e.
\eq[NalphaDeriv]
{
\Big(\frac{-N_\alpha'}{x^{\alpha-1}}\Big)' \le 0\,.
}
But
$$
\Big(\frac{N_\alpha'}{x^{\alpha-1}}\Big)' = \frac{x N_\alpha'' -(\alpha-1) N_\alpha'}{x^\alpha} =\frac{N_\alpha'''}{x^\alpha},
$$
where the last equality follows from \eqref{hermit}.

On the other hand, from \eqref{series} it follows that $N_\alpha''' =-\alpha N_{\alpha -2}'$. This expression is positive by Lemma \ref{root}. Hence \eqref{NalphaDeriv} is proved. This proves that
$$
u_\alpha \ge c_\alpha^\alpha - |x|^\alpha, \qquad x\in [-c_\alpha, c_\alpha].
$$
We conclude that the function $U$ from page \pageref{upq} majorizes the obstacle:
\eq[majorDavisObst]
{
U(p, q) \ge c_\alpha^\alpha |q|^\alpha -|p|^\alpha\,.
}

\subsection{Why constant $c_\alpha$ is sharp?}
\label{c_alpha_sharp}

The example, which show that the value $c_\alpha$ given on page \pageref{dex_u_alpha} cannot be replaced by larger value is  based on results of A.~Novikov  \cite{No}  and L.~Shepp  \cite{Shepp}.
Introduce the following stopping time
$$
T_a \df \inf\{t>0\colon |W(t)| =a\sqrt{t+1}\}, \quad a>0.
$$
It was proved  in \cite{Shepp}  that $\bE T_a^{\alpha/2} < \infty$ if 
$a < c_{\alpha}$ and that 
$\bE t_{c_\alpha}^{\alpha/2}=\infty$,  $\alpha > 0$. 
This gives us that $\bE t_{a}^{\alpha/2} \to \infty$,
 when $a\to c_\alpha-$.
  From here we get
  $$
  \lim_{a\to c_\alpha-} \frac{\bE (T_a+1)^{\alpha/2}}{\bE T_a^{\alpha/2}} =1\,.
  $$
  By definition of $T_a$  we have $|W(T_a)| =a\sqrt{T_a+1}$, and hence
  $$
  \lim_{a\to c_\alpha-} \frac{\bE|W(T_a)|^\alpha}{\bE T_a^{\alpha/2}} \to c_\alpha^\alpha\,.
  $$
  Now it follows immediately that the best constant  in \eqref{DavisBr01} cannot be larger than $c_\alpha$ defined on page \pageref{dex_u_alpha}. Davis in \cite{Davis} extended this estimate for the case of dyadic square function estimate \eqref{Davis01}.
  
\subsection{Why $U$ from page \pageref{upq} is the smallest  function satisfying \eqref{obstacle6} and \eqref{neravenstvo}?}
\label{the_smallest}
We know that on $\{ (p, q)\colon q\ge 0,\ |p|^\alpha \le c_\alpha^\alpha q^\alpha\}$ 
\begin{equation}
\label{bfUU}
|q|^\alpha c_\alpha^\alpha -|p|^\alpha \le \bfU(p, q) \le U(p, q)\,.
\end{equation}
Indeed, we proved that $U$ satisfies the main inequality and that it majorizes the obstacle $|q|^\alpha c_\alpha^\alpha -|p|^\alpha$. We also proved that
$\bf U$ is the smallest such function (this is true for any obstacle whatsoever). Hence, \eqref{bfUU} is verified.

But now we want to demonstrate that the Bellman function  is already found: $\bfU=U$. To do that we need to work a little bit more.

\bigskip

By definition on page \pageref{bfU6}  $\bfU$ is homogeneous of degree $\alpha$. 
We introduce $\bell(p)\df \bfU(p, 1)$, $b(p)\df U(p, 1)$. Thus we need to prove that
\eq[ravenstvo]
{
b(p)= \bell(p),\qquad p\in [-c_\alpha, c_\alpha].
}

One can easily rewrite \eqref{neravenstvo} in terms of $\bell$:  for all  $x\pm\tau  \in [-c_\alpha, c_\alpha]$  the following holds:
\eq[conc_bell6]
{
2\bell(x) \ge (1+\tau^2)^{\alpha/2} \bigg(\bell\Big(\frac{x+\tau}{\sqrt{1+\tau^2}}\Big)+\bell\Big(\frac{x-\tau}{\sqrt{1+\tau^2}}\Big)\bigg).
}
Since by construction $U(p, q) =0$ if $|q|^\alpha c_\alpha^\alpha -|p|^\alpha=0$ we conclude that $b(\pm c_\alpha)= \bell(\pm c_\alpha)=0$.

Combining \eqref{neravenstvo} with a simple observation that $\bfU$ by definition increases in $q$, we can conclude that function $\bf U$ is concave in $p$ for every fixed $q$, $\bell$ is concave. 

Let us recall that for any concave function $f$ the following holds (see e.g. \cite{EG}):
\begin{equation}
\label{concA}
f(x+h) = f(x) + f'(x) h + \frac12 f''(x) h^2 + o(h^2), \quad h\to 0, \quad \text{for a.e.}\,\,x\,.
\end{equation}

Then \eqref{concA} and inequality \eqref{conc_bell6} implies that $\bell''-x \bell'+\alpha \bell \le 0$ a.e.  But function $\bell$ is concave. 
In particular, it is everywhere defined and continuous, and its derivative $\bell'$ is also its distributional derivative, and it is everywhere defined
decreasing function.  

Let $(\bell)''$ denote the distributional derivative of decreasing function $\bell'$. Thus it is a non-positive measure. We denote its singular part by symbol $\sigma_s$. 
Hence, in the sense of distributions
\begin{equation}
\label{concD}
(\bell)'' - x \bell'\, dx+\alpha\bell \,dx= \big(\bell'' - x \bell' +\alpha\bell \big) \,dx + d \sigma_s \le 0\,.
\end{equation}


\begin{lemma}
\label{ode0}
Let $\alpha>0$. Let even non-negative concave function $v$  defined on $[-c_\al, c_\al]$ satisfy $v(\pm c_\al)=0$. Let $v$ satisfy $v''-x v'+\alpha v \le 0$ on $(-c_\al, c_\al)$ pointwisely and in the sense of distributions.  Assume also that
$v$ have finite derivative at $c_\al$: $v'(c_\al)>-\infty$.
Then $v''-x v'+\alpha v  = 0$ on $(-c_\al, c_\al)$ pointwisely and in the sense of distributions. Also $v= c u$ for some constant $c$. 
\end{lemma}

\begin{proof}
Let $u\df u_\al$ from \eqref{dex_u}. It is $C^2$ function and $u'' -x u' +\alpha u =0$ on $[-c_\al, c_\al]$.
Denote
$$
g\df v'' - x v' + \alpha v\,.
$$
Function $v$ is concave, so its second derivative is defined a.e., and we assumed that $g\le 0$.

Consider everywhere defined function
$$
w \df v' u -u' v\,.
$$
Its derivative is defined almost everywhere, and let us first calculate it  a.e.:
$$
w' = v'' u - u'' v= (g+ x v' -\al v) u - (x u' -\al u) v =g u + x w\,.
$$
Also in distributional sense
$$
(w)' = (v)'' u - u'' v\, dx = (g u + x w)\, dx + u \, d\sigma_s\,.
$$
Hence,
\begin{equation}
\label{Exp}
\frac{d}{dx} e^{-x^2/2} w= gu e^{-x^2/2}, \quad \text{for almost every}\,\, x\,,
\end{equation}
and 
\begin{equation}
\label{ExpD}
\Big(e^{-x^2/2} w\Big)' = gu e^{-x^2/2}\, dx + u  e^{-x^2/2}\, d\sigma_s, \,\, \text{in distribution sense}\,.
\end{equation}
Measure $\sigma_s$ is non-positive, therefore, these two inequalities \eqref{Exp}, \eqref{ExpD} mean that
for any two points $0<a<b <1$ we have
$$
e^{-b^2/2} w(b) - e^{-a^2/2} w(a) \le \int_a^b gu e^{-x^2/2}\, dx,
$$
moreover, the inequality is strict, if $\sigma_s(a, b) \neq 0$.

Let us tend $b$ to $1$. Looking at the definition $w= v'u -u'v$ and using the assumptions of lemma, we conclude
that $e^{-b^2/2} w(b)\to 0$.  Hence, 
\begin{equation}
\label{a1exp}
 e^{-a^2/2} w(a) \ge \int_a^1 (-g)u e^{-x^2/2}\, dx\,.
\end{equation}
Again the  inequality is strict if $\sigma_s(a, 1) \neq 0$.

Now let us tend $a\to 0$. By smoothness and evenness $u'(a)\to 0$. But $u(a)>0$ and $v'(a) \le 0$ for  a. e. $a>0$. Therefore,
$$
\limsup_{a\to 0+}  e^{-a^2/2} w(a) \le 0\,.
$$
Combining this with \eqref{a1exp} we conclude that
$$
\int_a^1 (-g)u e^{-x^2/2}\, dx \le 0
$$
with the strict inequality if $\sigma_s(0, 1) \neq 0$. The strict inequality is of course leads to contradiction (recall that $-g\ge 0, u>0$), so we conclude
that $\sigma_s$ is a zero measure on $(0,1)$. But also even a non-strict inequality implies that $g=0$ a.e.

We conclude from \eqref{Exp}, \eqref{ExpD} that $ e^{-x^2/2} w(x)$ is constant on $(0, 1)$. But we already saw that 
this function tends to $0$ when $x$ tends to $1$. Thus, identically on $(0,1)$
$$
u' v- v' u = w =0\,.
$$
This means that $v/u = const$. Lemma is proved.

\end{proof}

\bigskip

Now it is easy to prove \eqref{ravenstvo}: $\bell=b$.  Choose $v=\bell$, the  assumptions on ordinary differential inequality is easy to verify, see \eqref{concD}.  Of course this function vanishes at $\pm c_\al$.  Also by the definition of $b$ it is clear  (see \eqref{dex_u}, \eqref{upq}) that
$$
b'(c_\al)= -\al c_\al^{\al-1}>-\infty\,.
$$
We are left to see that  the same is true for $\bell'(c_\al)$. 

Recall that $\bell(\cdot)= {\bf U}(\cdot, 1)$,   $b(\cdot) = U(\cdot, 1)$, then by \eqref{bfUU}  we definitely know that 
$$
c_\al^\al - |x|^\al \le \bell(x) \le b(x),\quad x \in [-c_\al, c_\al]\,.
$$
The functions on the left and on the right vanish at $c_\al$ and have the same derivative $ -\al c_\al^{\al-1}$ at $C_\al$. Hence,
$\bell$ is in fact differentiable at $c_\al$ (the left derivative exists), and its (left)  derivative satisfies
$$
\bell'(c_\al)=b'(c_\al)= -\al c_\al^{\al-1}>-\infty\,.
$$
But now Lemma \ref{ode0} says that $\bell = const\cdot b$. Since we have the above relationship on derivatives, the constant has to be $1$. 
We proved \eqref{ravenstvo}. This gives 
$$
{\bf U} = U,
$$
where $U$ was defined in \eqref{dex_u}, \eqref{upq}. We found the Bellman function $\bf U$ for Burkholder--Gundy--Davis inequality, and we completely solved the obstacle problem with the obstacle $O(p, q)= c_\al^{\al} q^\al- |p|^\al$, $\al\ge 2$.

\subsection{When obstacle coincides with its heat envelope}
\label{heat_env_coincide}
The next corollary immediately follows from the previous proposition, and it describes one possibility when the heat envelope  coincides with its obstacle 
\begin{cor}
Let $O(p,q)$ be $C^{2}(\mathbb{R}\times[0,\infty))$ obstacle such that 
\begin{align*}
&O_{pp} + \frac{O_{q}}{q}\leq 0 \quad \text{and}\\
&t \mapsto O(p,\sqrt{t}) \quad \text{is convex}. 
\end{align*}
Then  the heat envelope $U$ of $O$ satisfies $U(p,q)=O(p,q)$.
\end{cor}

The next proposition says that if $O$ satisfies ``backward heat equation'' then the convexity assumption $ t\mapsto O(p,\sqrt{t})$ is necessary and sufficient for main  inequality \eqref{MISq}. 
\begin{prop}
Let $O(p,q) \in C^{4}(\mathbb{R}\times[0,\infty))$ be such that
\begin{align*}
O_{pp}+\frac{O_{q}}{q}=0
\end{align*}
for all $(p,q)\in \mathbb{R}\times(0,\infty)$. Then the following conditions are equivalent 
\begin{itemize}
\item[(i)] The map $t\mapsto O(p,\sqrt{t})$ is convex for $t\geq 0$. 
\item[(ii)]
$
2O(p,q) \geq O(p+a,\sqrt{q^{2}+a^{2}})+O(p-a,\sqrt{q^{2}+a^{2}})
$
for all $p,a \in \mathbb{R}$ and all $q\geq 0$. 
\end{itemize}
\end{prop}
\begin{proof}
The implication $(i) \Rightarrow (ii)$ follows from Lemma \ref{barko}. It remains to  show the implication $(ii)\Rightarrow (i)$. By Taylor's formula as $a \to 0$ we have 
\begin{align*}
&O(p+a,\sqrt{q^{2}+a^{2}})+O(p-a,\sqrt{q^{2}+a^{2}}) 
\\
&=2O(p,q)+\left(\!O_{pp}+\frac{O_{q}}{q} \!\right) a^{2}+\left(\!O_{pppp}+6\frac{O_{ppq}}{q}+3\frac{O_{qq}}{q^{2}}-3\frac{O_{q}}{q^3}\!\right) \frac{a^{4}}{12}+o(a^{4})
\end{align*}
Since $O_{pp}+\frac{O_{q}}{q}=0$ we see that 
\begin{align*}
O_{pppp}+6\frac{O_{ppq}}{q}+3\frac{O_{qq}}{q^{2}}-3\frac{O_{q}}{q^3}=2\left(\frac{O_{q}}{q^3}-\frac{O_{qq}}{q^{2}} \right).
\end{align*}
Therefore,
\begin{align*}
0&\geq O(p+a,\sqrt{q^{2}+a^{2}})+O(p-a,\sqrt{q^{2}+a^{2}})-2O(p,q)=
\\
&=\left(\frac{O_{q}}{q}-O_{qq}\right)\frac{a^{4}}{6q^{2}}+o(a^{4}).
\end{align*}
Thus we obtain that $\frac{O_{q}}{q}-O_{qq}\leq 0$. On the other hand the latter inequality is equivalent to the fact that $t \to O(p,\sqrt{t})$ is convex.  
\end{proof}
 

\section{Bollob\'as function}
\label{bol}

This part of the present article is taken from \cite{HoIvV}. We put it here because the solution of the obstacle problem(s) in this section and the solution
of the obstacle problem in the previous section have so much in common, and at the same time, they have essential differences. So we include the current section for the sake of comparison.

The classical Littlewood--Khintchine inequality states  that
\begin{equation}
\label{LKh}
\big(\sum_{k=1}^n a_k^2\big)^{1/2} \le L \int_0^1 \big| \sum_{k=1}^n a_k r_k(t) dt\big|,
\end{equation}
where $\{r_k(t)\}$ are Rademacher functions. It was one of Littlewood's problem to find the best value for constant $L$. The problem was solved by S.~Szarek~\cite{Sza}, see also~\cite{Hall}. The sharp constant is $L=\sqrt{2}$.

B. Bollob\'as \cite{Bollobas} considered the following related problem, which we formulate in the form convenient for us. The problem of Bollob\'as was: what is the best value for the constant $B$  for the following inequality
\begin{equation}
\label{BolIneq}
\lambda \big|\{ t\in (0,1)\colon Sf(t)\ge \lambda\}\big| \le C\|f\|_1\,?
\end{equation}

Consider $x_n\df\sum_{k=1}^n a_k r_k(t)$. If we denote $\lambda\df  \big(\sum_{k=1}^n a_k^2\big)^{1/2}= Sx_n(x)$ (obviously $Sx_n$ is a constant function), we get

\begin{align}
\label{laBol}
&\big(\sum_{k=1}^n a_k^2\big)^{1/2}\!\!\! =\lambda  \big|\{ t\in (0,1)\colon Sx_n(t)\ge \lambda\}\big|  \le \nonumber\\ 
& C\int_0^1 \big| \sum_{k=1}^n a_k r_k(t)\big| dt\,.
\end{align}

This says that $\sqrt{2}= L\le B$. On the other hand, D. Burkholder in \cite{BuMart} proved that $B\le 3$. B. Bollob\'as in \cite{Bollobas} conjectured the best value of $B$, and in 2009 A. Os\c ekowski \cite{Os2009} proved this conjecture. We will give a slightly different proof by solving the obstacle problem and finding the heat envelopes of two obstacles:
\begin{align}
&O_1(p, q) = {\bf 1}_{q\ge 1} - C_1 |p|, \label{O1} 
\\
&O_2(p, q) = {\bf 1}_{p^2+q^2\ge 1} - C_2 |p|. \label{O2}
\end{align}

We are interested in the smallest possible values of $C_1$ and $C_2$ such that these functions have (finite)  heat envelopes.
The reader will see, in particular, that $C_1=C_2=C$ and that the heat envelopes of these two functions coincide.

\medskip 

Define the following Bellman function:
	\eq[defB6]
	{
	\Bel(x, \lambda) \df   \inf\{\av{|\varphi|}J\colon \av{ \varphi}J = x; \ 
		S_J^2\varphi \geq \lambda \ \text{ a.\,e. on } J\}.
	}
Some of the obvious properties of $\Bel$  are:
	\begin{itemize}
	\item Domain: $\Omega_{\Bel} \df  \{(x, \lambda)\colon x\in\bR; \,\,\lambda > 0\}$;
	\item $\Bel$ is increasing in $\lambda$ and even in $x$;
	\item Homogeneity: $\Bel(tx, t^2\lambda) = |t| \Bel(x, \lambda)$;
	\item Range/Obstacle Condition: $|x| \leq \Bel(x, \lambda) \leq \max\{|x|, \sqrt{\lambda}\}$;
	\item Main Inequality:
		\begin{equation} 
		\label{E:L-MI}
		2\Bel(x, \lambda) \leq \Bel(x-a, \lambda - a^2) + \Bel(x + a, \lambda - a^2), \:\: \forall \: |a| < \sqrt{\lambda}.
		\end{equation}
	\item $\Bel$ is convex in $x$, and so it is easy to see that $\Bel$ is minimal at $x = 0$:
		\begin{equation} \label{E:L-min0}
		\Bel(0, \lambda) \leq \Bel(x, \lambda), \:\: \forall \: x,
		\end{equation}
	therefore  we can use that $\Bel$ is increasing in $\lambda$ and also use the minimality at $x=0$ to obtain from \eqref{E:L-MI} that $\Bel$ is non-decreasing in $x$ for $x\geq 0$, and non-increasing in $x$ for $x\leq 0$;
	\item Greatest Subsolution: If $B(x, \lambda)$ is any continuous non-negative function on $\Omega_{\Bel}$ which satisfies the main inequality \eqref{E:L-MI} and the range condition $B(x, \sqrt{\lambda}) \leq \max\{|x|, \sqrt{\lambda}\}$, then
		$B \leq \Bel$. 
	\end{itemize}

\medskip

\subsection{Bellman induction}
\label{Bell_Ind_Boll}
	
\begin{theorem}
\label{T:bL-GrtSub}
If $B$ is any subsolution as defined above, then \hbox{$B \leq \Bel$}.
\end{theorem}

\begin{proof}
We must prove that $B(x, \lambda) \leq \av{|\varphi|}J$ for any function $\varphi$ on $J$ with $\av{\varphi}J = x$,
$|J| = |\{x\in J \colon S_J^2\varphi(x) \geq \lambda\}|$. As before, we may assume that there is some dyadic level $N \geq 0$ below which the Haar coefficients of $\varphi$ are zero.

If $\lambda \leq (\Delta\cii J\varphi)^2$, then by the range/obstacle condition above
	$$ B(x,  \lambda) \leq \max\{|x|, \sqrt{\lambda}\} \leq \max\{|x|, |\Delta\cii J\varphi|\} \leq \av{|\varphi|}J,$$
and we are done. Otherwise, put $\lambda_{J_{\pm}} = \lambda - (\Delta\cii J\varphi)^2 > 0$, $x_{J_{\pm}} = \av{\varphi}{J_{\pm}}$.
Then by the main inequality:
	$$|J|B(x,\lambda) \leq |J_{-}|B(x_{J_{-}},  \lambda_{J_{-}}) +
		|J_{+}|B(x_{J_{+}},  \lambda_{J_{+}}).$$
If $\lambda_{J_{-}} \leq (\Delta_{J_{-}}\varphi)^2$, it follows as before that 
	$|J_{-}|B(x_{J_{-}}, \lambda_{J_{-}}) \leq \int_{J_{-}}|\varphi|$,
and otherwise we iterate further on $J_{-}$.

Continuing this way down to the last level $N$ and putting 
	$\lambda_I \df  \lambda - (\Delta_{I^{(1)}}\varphi )^2- \ldots - (\Delta\cii J\varphi)^2$
for every $I \in \cD_N(J)$, where $I^{(1)}$ denotes the dyadic father of $I$, the previous iterations have covered all cases where $\lambda_I \leq 0$, and we have (with $x_I\def  \av{\varphi}{I}$)
	\begin{equation}\label{E:bL-GrtSub-temp} 
	|J|B(x, \lambda)\  \leq \sum_{\genfrac{}{}{0pt}{}{I\in \cD_{N}(J)}{ \lambda_I \leq 0}} \int_I |\varphi| + 
		\sum_{\genfrac{}{}{0pt}{}{I \in \cD_{N}(J)} {\lambda_I > 0}} \!\!|I| B(x_I,  \lambda_I).
	\end{equation}
Now note that for all $I \in \cD_N(J)$ we must have
 $\lambda_I \leq (\Delta\cii I\varphi)^2$ just because $S_J^2\varphi(x) \geq \lambda$ everywhere on $J$, so we use condition the range/obstacle condition as before to obtain
	$B(x_I, \lambda_I) \leq \max\{|x_I|, |\Delta\cii I\varphi|\} \leq \av{|\varphi|}J.$
Finally, \eqref{E:bL-GrtSub-temp} becomes:
	$$
	|J|B(x, \lambda) \ \leq \sum_{I \in \cD_N(J)}\int_I |\varphi| = \int_J |\varphi|.
$$
	This finishes the proof of the claim
	$$
	B \le \Bel\,.
	$$
\end{proof}

\subsection{Finding the candidate for  $\Bel(x,  \lambda)$}

We introduce 
$$
\bell(\tau)\df \Bel(\tau, 1)\,.
$$
Using homogeneity, we write
	$$  
	\sqrt{\lambda} \bell(\tau)=  \sqrt{\lambda} \Bel\left(\frac{x}{\sqrt{\lambda}} , 1\right) =\Bel(x, \lambda)\,,\,\,\, 
	\text{where } \tau=\frac{x}{\sqrt{\lambda}}\,. 
	$$
Then $\bell\colon \bR \rightarrow [0, \infty)$, $\bell$ is even in $\tau$, and from \eqref{E:L-min0}:
	\begin{equation}\label{E:beta-min0}
		\bell(0) \leq \bell(\tau), \qquad\forall \: \tau.
		\end{equation}
Moreover, $\bell$ satisfies
		\begin{equation}\label{E:beta-OC}
		\bell(\tau) = |\tau|, \qquad\forall \: |\tau| \geq 1\,.
		\end{equation}
		
		\medskip

		We are looking for a candidate $B$  for $\Bel$. We will assume  now that $\Bel$ is smooth. 
		We will find the candidate under this assumption, and later we will prove that thus found function is indeed $\Bel$.
Using again Taylor's formula, the infinitesimal version of \eqref{E:L-MI} is
	\begin{equation}
	\label{infiniL}
	{\Bel}_{xx} - 2{\Bel}_{\lambda} \geq 0.
	\end{equation}
In terms of $\bell$, this becomes
	\begin{equation}\label{E:beta-DI}
	\bell''(\tau) + \tau \bell'(\tau) - \bell(\tau) \geq 0.
	\end{equation}
Since $\bell$ is even, we focus next only on $\tau \geq 0$. 

\medskip

Let  symbol $\Phi$ denote the following  function:
	$$
	\Phi(\tau)\df \int_0^\tau e^{-y^2/2} \, dy\,.
	$$
	Put
	$$
\Psi(\tau) = \tau\Phi(\tau) + e^{-\tau^2/2}, \qquad\forall \: \tau \geq 0.
$$
The general solution of the differential equation 
$$
z''(\tau) + \tau z'(\tau) - z(\tau) = 0, \qquad \tau \geq 0
$$ 
is
	$$
	z(\tau) = C \Psi(\tau) + D\tau\,.
	$$
	
Note that
	\begin{equation}
	\label{Phi}
	\Psi'(\tau) =\Phi(\tau),\qquad \Psi''(\tau) = e^{-\tau^2/2}\,.
	\end{equation}
Since $\bell(\tau) = \tau$ for  $\tau \geq 1$,  see \eqref{E:beta-OC},
a reasonable candidate for our function $\bell$ is one already proposed by B. Bollobas \cite{Bollobas}:
	\begin{equation}\label{E:beta-candidate}
	b(\tau) \df  \left\{ \begin{array}{ll}
		\frac{\Psi(\tau)}{\Psi(1)}, &\quad 0 \leq \tau < 1
		\\
		\tau, &\quad \tau \geq 1.
		\end{array}\right.
	\end{equation}
In other words, a candidate for $\Bel$ is
	\begin{equation}
	\label{L1}
	B(y, \la)=\begin{cases}\sqrt{\lambda} \frac{\Psi\left(\frac{|y|}{\sqrt{\lambda}}\right)}{\Psi(1)},\quad&\sqrt{\la}\ge |y|,
	\\
	|y|,\qquad &\sqrt{\la}\le |y|\,.\end{cases}
	\end{equation}

Our first goal  will be to go from differential  inequality  \eqref{infiniL} to its finite difference version \eqref{MISq}. 

\begin{lemma}
\label{subs-2}
The function $B$ defined in \eqref{L1} satisfies the  finite difference main inequality \textup(the analog of \eqref{E:L-MI}\textup)\textup:
\begin{equation}
\label{MISqL}
2 B(y, \lambda) \le B(y-a, \lambda-a^2) + B(y+a, \lambda-a^2)\,.
\end{equation}
\end{lemma}

We already saw in Lemma \ref{barko} that under some extra assumptions of convexity one can 
derive the finite difference inequalities from their differential form (infinitesimal form). 
Unfortunately, this approach will not work for function $B$ defined in \eqref{L1}. 
This function exactly misses the extra property \eqref{amoz} of Lemma \ref{barko}.
In fact, we deal now with convexity paradigm  rather than concavity conditions of Lemma \ref{barko}, so the right analog of 
property \eqref{amoz} for $B$ in the above formula would be
$$
\lambda \to B(y, \lambda)\,\, \text{is a concave function for every fixed}\,\,  y\,.
$$

But it is obvious that our candidate  $B$ does not have this property. This is why the proof of Lemma \ref{subs-2} requires direct calculations.
This requires splitting the proof into several cases. One of them was considered in \cite{Bollobas}, but other cases were only mentioned there.

\begin{proof}[Proof of Lemma \ref{subs-2}]
By symmetry we can think that $x\ge  0$. Case 1) will be when  both points $x\pm t, \la-t^2)$ lie in $\Pi$ (i.\,e. they lie over parabola $\la=x^2$).


\bigskip

Case 1). We follow \cite{Bollobas}. Put 
\begin{equation}
\label{X}
X(x, \tau)\df  \frac{x+\tau}{(1-\tau^2)^{1/2}},\,\, \tau\in [0, x],\,\, x\in [0, 1)\,.
\end{equation}
In our case \eqref{MISqL}  can be rewritten as \hbox{($\tau\df  a/\sqrt{\la},$ \, $ x=y/\sqrt{\la}$)}:
\begin{equation}
\label{bo1}
2\Psi(x) \le \Psi(X(x, \tau)) +\Psi(X(x, -\tau)),
\end{equation}
which is correct for $\tau=0$. Let us check that 
\begin{equation}
\label{bo2}
\frac{\pd}{\pd \tau} (\Psi(X(x, \tau)) +\Psi(X(x, -\tau))) \ge 0\,.
\end{equation}
Using \eqref{Phi}, we get the equality
\begin{align*}
&\frac{\pd}{\pd \tau} (\Psi(X(x, \tau)) +\Psi(X(x, -\tau))) = \frac1{1-\tau^2} ( \Phi(X(x, \tau)) - \Phi(X(x, -\tau))) \\
& \frac{x\tau}{1-\tau^2}( \Phi(X(x, \tau)) + \Phi(X(x, -\tau))) - \frac{\tau}{(1-\tau^2)^{1/2}}(X(x, \tau)) \Phi(X(x, \tau))   
\\
&+X(x, -\tau))\Phi(X(x, -\tau)))-\frac{\tau}{(1-\tau^2)^{1/2}}( e^{-X(x, \tau)^2/2} +  e^{-X(x, -\tau)^2/2} )\,.
\end{align*}
After plugging  \eqref{X} this simplifies to
\begin{align*}
&\frac{\pd}{\pd \tau} (\Psi(X(x, \tau)) +\Psi(X(x, -\tau))) =( \Phi(X(x, \tau)) - \Phi(X(x, -\tau))) 
\\
&-\frac{\tau}{(1-\tau^2)^{1/2}}( e^{-X(x, \tau)^2/2} +  e^{-X(x, -\tau)^2/2} )\,.
\end{align*}
But $\frac{\tau}{(1-\tau^2)^{1/2}} = \half  (X(x, \tau)- X(x, -\tau))$, so  to prove \eqref{bo2} one needs to check the following inequality.
\begin{equation}
\label{bo3}
\dashint_{X(x, -\tau)}^{X(x, \tau)} e^{-s^2/2} ds \ge \half ( e^{-X(x, \tau)^2/2} + e^{-X(x, -\tau)^2/2})\,.
\end{equation}
This inequality holds because in our case 1) we have $X(x, -\tau)\in [-1,1], X(x, \tau)\in [-1,1]$, 
and function $s\to e^{-s^2/2}$ is concave on the interval $[-1,1]$.
(It is easy  that for every concave function on an interval, its average over the interval is at least its average over the ends of the interval.)

Case 2). Now suppose that the left point $(x-t, \la-t^2)$ lies on parabola. 
By homogeneity we can always think that $\la=1$.  We continue to consider by symmetry $x\ge 0$ only.
If $(x-t,1-t^{2})$  is such that  $(x-t)^{2}=1-t^{2}$ then we need to show that 
\begin{align}\label{patulya1}
2 \frac{ \Psi(x)}{\Psi(1)}\leq 2t
\end{align}
Clearly $0\leq t\leq 1$, $0\leq x\leq t$. From $(x-t)^{2}=1-t^{2}$ 
we obtain that $t-\sqrt{1-t^{2}}\df x(t) \geq 0,$ so $t\geq \frac{1}{\sqrt{2}}$, and the inequality (\ref{patulya1}) simplifies to 
$$
x(t)\leq \Psi^{-1}(\Psi(1)t), \quad  \frac{1}{\sqrt{2}}\leq t\leq 1.
$$
The left hand side is convex and the right hand side is concave. 
Since at $t=1$ and $t=\frac{1}{\sqrt{2}}$ the inequality holds then it holds on the whole interval $[1/\sqrt{2},1]$. 

So we proved that if the left point already left $\Pi$ (and then automatically the right point also already left it), the desired inequality holds. 

Case 3). It   remains to show that  if the right point already left $\Pi$ but the left point is in $\Pi$, then \eqref{MISqL} still holds. 
Again by homogeneity we can always think that $\la=1$. Then the required inequality amounts to 
\begin{align*}
2\Psi(x) \leq \sqrt{1-t^{2}} \Psi\left(\frac{t-x}{\sqrt{1-t^{2}}}\right)+\Psi(1)(x+t)
\end{align*}
where either $\sqrt{1-t^{2}}-t \leq x\leq t\leq 1$ or $\sqrt{1-t^{2}}-t \leq t\leq x\leq 1$ . It is the same as to show 
\begin{align}\label{erti1}
 \Psi\left(\frac{t-x}{\sqrt{1-t^{2}}}\right) + \left(\frac{t-(\frac{2\Psi(x)}{\Psi(1)}-x)}{\sqrt{1-t^{2}}}\right)\Psi(1) \geq  0
\end{align}
for all $0\leq x\leq 1$ if $ \frac{\sqrt{2-x^{2}}-x}{2}< t< \frac{x+\sqrt{2-x^{2}}}{2}$. 
The left inequality says that the right point already crossed parabola $\pd\Pi$  and the right inequality says that the left point is still inside  $\Pi$.

 Let as show that the derivative in $t$ of the left hand side of (\ref{erti1}) is nonnegative. 
 If this is the case then we are done. $\Psi$ is increasing (see \eqref{Phi}), and since 
 $xt\leq 1$ therefore $t \mapsto \Psi\left(\frac{t-x}{\sqrt{1-t^{2}}}\right)$ is increasing 
 as a composition of two increasing functions. By the same logic, to check the monotonicity  
 of the map $t\mapsto \frac{t-(\frac{2\Psi(x)}{\Psi(1)}-x)}{\sqrt{1-t^{2}}}$ 
 it is enough to verify that $t(\frac{2\Psi(x)}{\Psi(1)}-x) \leq 1$. The latter inequality follows from the following two simple inequalities 
\begin{align}
&\Psi(x) \geq  \Psi(1) x, \quad 0\leq x\leq 1
\label{ori}
\\
&\left(\frac{x+\sqrt{2-x^{2}}}{2} \right) \left(\frac{2\Psi(x)}{\Psi(1)}-x\right)\leq 1, \quad  0\leq x\leq 1 \label{sami}
\end{align}
Indeed, to verify (\ref{ori}) notice that  
\eq[ori6]
{
\frac{d}{dx}\frac{\Psi(x)}{x} = \frac{x\Phi(x)-\Psi(x)}{x^{2}} = -\frac{e^{-\frac{x^{2}}{2}}}{x^{2}} <0,
}
therefore $\frac{\Psi(x)}{x}\geq \Psi(1)$ when $0\le x\le 1$. 

To verify (\ref{sami}) it is enough to show that 
$$
\frac{\Psi(x)}{\Psi(1)x}\leq \frac{1}{x^{2}+x\sqrt{2-x^{2}}}+\frac{1}{2}
$$
If $x=1$ we have equality. Taking derivative of the mapping 
$$
x \to  \frac{\Psi(x)}{\Psi(1)x} -\frac{1}{x^{2}+x\sqrt{2-x^{2}}}-\frac{1}{2}
$$
 we obtain 
$$
\frac{2}{x^{2}}\left(-\frac{e^{-\frac{x^{2}}{2}}}{2\Psi(1)}+\frac{x+\frac{1-x^{2}}{\sqrt{2-x^{2}}}}{(x+\sqrt{2-x^{2}})^{2}}\right)\geq 0
$$
To prove the last inequality it is the same as to show that 
$$
\frac{\sqrt{2-x^{2}}+x(2-x^{2})}{x\sqrt{2-x^{2}}+1-x^{2}}\leq \Psi(1) e^{\frac{x^{2}}{2}}\,.
$$
For the exponential function we use the estimate $e^{\frac{x^{2}}{2}}\geq 1+\frac{x^{2}}{2}$. 
We estimate $\sqrt{2-x^{2}}$ from above in the numerator by $\sqrt{2}(1-\frac{x^{2}}{4})$, 
and we estimate $\sqrt{2-x^{2}}$ from below in the denominator by $(1-\sqrt{2})(x-1)+1$ 
(as $x\to\sqrt{2-x^{2}}$ is concave). Thus it would be enough to prove that 
$$
\frac{\sqrt{2}(1-\frac{x^{2}}{4})+x(2-x^{2})}{\sqrt{2}x(1-x)+1}\leq \Psi(1) \left(1+\frac{x^{2}}{2} \right), \quad 0\leq x\leq 1
$$  
If we further use the estimates  $\Psi(1) > \frac{29}{28}$,  
and $\frac{41}{29}< \sqrt{2}< \frac{17}{12}$ (for  denominator and numerator correspondingly), then the last  inequality would follow  from  
$$
\frac{29}{240} \cdot \frac{246x^{4}-486x^3+233x^{2}-12x-8}{29+41x-41x^{2}}\leq 0
$$
The denominator has the positive sign. The negativity of $246x^{4}-486x^3+233x^{2}-12x-8\leq 0$ 
for $0\leq x\leq 1$ follows from the Sturm's algorithm,  
which shows that the polynomial does not have roots on $[0,1]$. Since at point $x=0$ it is negative therefore it is negative on the whole interval.  
\bigskip 
\end{proof}

\subsection{Finding $\Bel$}
\label{FbL}
Since it is easy to verify that $B$ satisfies the range condition 
$B(x, \lambda) \leq \max\{|f|, \sqrt{\lambda}\}$, we have then that $B$ is a subsolution of \eqref{MISqL}, and so, by Theorem \ref{T:bL-GrtSub}
	$$
	B \leq \Bel\,.
	$$
	
	\medskip
	
	Now we want to prove the opposite inequality
\begin{equation}
\label{bLsm}
\Bel\le B\,.
\end{equation}

\begin{lemma}
\label{ode1}
Let even functions $\bell$ and $b$ defined on $[-1,1]$ satisfy $\bell(1)=b(1)=1$, and $b'' +xb'-b=0$, $b\in C^2$, $\bell$ being a convex function such that $\bell''+x \bell'-\bell \ge 0$ on $(-1,1)$  in the sense of distributions. 
Then $\bell\le b$.
\end{lemma}

\begin{proof}
If $\bell$ were in $C^2$ as well, then this would be very easy. In fact, consider $a(x)\df  \bell(x)-b(x)$. At end points it is zero, and $a''+xa'-a\ge 0$. Assume that function  $a$ is strictly positive somewhere, then it should have a maximum, where it is positive. Let it be $x_0$. Then $a(x_0)>0, a'(x_0) =0$. So $a''(x_0) \ge a(x_0)>0$. Then  $x_0$ cannot be maximum, so we come to a contradiction.

If $\bell$ is not $C^2$, we still consider $a(x)\df  \bell(x)-b(x)$, which is still a continuous function 
on $[-1, 1]$ equal to $0$ at the endpoints. If it is positive somewhere, it should have a positive maximum,  let $s_0$ be a point of maximum.

Since $\bell$ is assumed to be convex, function $a'$ is of bounded variation, and as such it is the sum of $f$ and $g$, 
where $f$ is a continuous function and $g$ is a jump function. Notice that 
1) all jumps are positive, as they came only from $\bell$, and  b) $g$ is continuous everywhere except the  countable set of jump points.


As $a'$ is a function of bounded variation it has one-sided limits at any interior point. Let $a'(s_0\pm)$ be right and left limits correspondingly. Since all the jumps are positive we have
$$
 a'(s_0+) \ge a(s_0-).
$$
But $s_0$ is a point of maximum of $a$, so $ a'(s_0-) \ge 0$, $ a'(s_0+) \le 0$.
All together may happen only if $a'(s_0+) = a'(s_0-)=0$. But this means that
$s_0$ is not a jump point.

 By continuity at $s_0$, $ a'$ is  small near $s_0$, but $a(s_0)>0$, so we can choose a small neighborhood of $s_0$, where $|sa'(s)|< \half a(s)$.

Since $a'' +s a'-a\ge 0$,  in this neighborhood of 
$s_0$ we have 
$$
a''\ge a-sa'> \half a \ge 0
$$ 
in the sense of distributions.  But a convex function cannot have maximum strictly inside an interval. We come to a contradiction.

 Lemma is proved.
 \end{proof}

We found the Bellman function $\Bel$, the formula is given in the following theorem.

\begin{theorem}
\label{bLthm0}
\begin{equation}
	\label{L11}
	\Bel(x, \la)=
	\begin{cases}
	\sqrt{\lambda} \frac{\Psi\left(\frac{|x|}{\sqrt{\lambda}}\right)}{\Psi(1)},\quad  &x^2\le \la,
	\\
	\quad |x|,\quad & x^2\ge \la\,.
	\end{cases}
	\end{equation}
	\end{theorem}
	
Let us introduce an obstacle function defined on $\bR^2$.
\begin{equation}
\label{Oobst}
O(x, \la)\df \begin{cases}  |x|, \quad x^2\ge \la\\
\infty,\quad x^2<\la\,.\end{cases}
\end{equation}

\begin{theorem}
\label{bLthm1}
Function $\Bel$
is the largest function satisfying the finite difference inequality
such that  it is majorized by  the obstacle function $O(x, \la)$:
\begin{equation}
\label{MILob}
\Bel(x, \la) \le  O(x, \lambda)\,.
\end{equation}
Moreover,
\begin{equation}
\label{max1}
\Bel(x, \la) =\max\bigg(\sqrt{\la}\frac{\Psi(\frac{x}{\sqrt{\la}})}{\Psi(1)},  |x|\bigg)\,.
\end{equation}
\end{theorem}


\begin{thebibliography}{999}

\bibitem{BM} {\sc F.~Barthe, B.~Maurey}, {\em Some remarks on isoperimetry of Gaussian type}, Annales de l'Institut Henri Poincare (B) Probability and Statistics, Vol. 36, Iss. 4, pp. 419--434.

\bibitem{Bollobas} {\sc B.~Bollob\'as}, {\em Martingale Inequalities}, Math. Proc. Camb. Phil. Soc. Vol. 87, 1980, pp. 377--382.


\bibitem{Buck1}{\sc St. Buckley}, {\em Estimates for operator norms on weighted spaces and reverse Jensen inequalities}.
Trans. Amer. Math. Soc. 340 (1993), no. 1, 253--272.

\bibitem{Bu1} {\sc D. Burkholder}, {\em A geometrical characterization of Banach spaces 
in which martingale difference sequences are unconditional}, Ann. Probab. 9 (1981) 997--1011.

\bibitem{Bu2} {\sc D. Burkholder}, {\em Boundary value problems and sharp estimates for the 
martingale transforms}, Ann. of Prob. 12 (1984) 647--702.

\bibitem{Bu3} {\sc D. Burkholder}, {\em Martingales and Fourier analysis in Banach spaces}, 
Probability and Analysis (Varenna, 1985) Lecture Notes in Math. 1206 (1986) 61--108.

\bibitem{Bu4} {\sc D. Burkholder}, {\em An extension of classical martingale inequality}, 
Probability Theory and Harmonic Analysis, ed. by J.-A. Chao and W. A. Woyczy\'nski, Marcel Dekker, 1986.

\bibitem{Bu5} {\sc D. Burkholder}, {\em Sharp inequalities for martingales and stochastic 
integrals}, Colloque Paul L\'evy sur les Processus Stochastiques (Palaiseau, 1987), 
Ast\'erisque No.~157-158 (1988) 75--94.

\bibitem{Bu6} {\sc D. Burkholder}, {\em A proof of the Pelczynski's conjecture for the Haar system}, 
Studia Math. 91 (1988) 79--83.

\bibitem{Bu7} {\sc D. Burkholder}, {\em Differential subordination of harmonic functions 
and martingales}, (El Escorial 1987) Lecture Notes in Math. 1384 (1989) 1--23.

\bibitem{Bu8} {\sc D. Burkholder}, {\em Explorations of martingale theory and its applications}, 
Lecture Notes in Math. 1464 (1991) 1--66.

\bibitem{Bu9} {\sc D. Burkholder}, {\em Strong differential subordination and stochastic 
integration}, Ann. of Prob. 22 (1994) 995--1025.

\bibitem{Bu10} {\sc D. Burkholder}, {\em Martingales and Singular Integrals in Banach spaces}, 
Handbook of the Geometry of Banach Spaces, Vol. 1, Ch. 6., (2001) 233--269.

\bibitem{BuMart} {\sc D. L. Burkholder}, {\em Martingale transforms}, Ann. Math. Statist. 37 (1966), 1494--1504.

\bibitem{BuGa} {\sc D. L. Burkholder and R. F. Gundy}, {\em Extrapolation and interpolation of quasi-linear operators on martingales}, Acta Math. 124 (1970), 249--304.


\bibitem{Davis} {\sc B. Davis}, {\em On the $L^p$ norms of stochastic integrals and other martingales}, Duke Math. J. 43 (1976),697--704.

\bibitem{EG} {\sc L. C. Evans, R. F. Gariepy}, 
{\em Measure Theory and Fine  Properties of Functions.}, 1992.


\bibitem{Hall} {\sc R. R. Hall}, {\em  On a conjecture of Littlewood},  Math. Proc. Cambridge Philos. Soc., 78 (1975),
443--445.

\bibitem{HoIvV} {\sc I. Holmes, P. Ivanisvili, A. Volberg}, {The Sharp Constant in the Weak (1,1) Inequality for the Square Function: a new proof}, arXiv:1710.01346, pp. 1--17. 

\bibitem{Mar} {\sc J. Marcinkiewicz}, {\em Quelque th\'eor\`emes sur les s\'eries orthogonales}, 
Ann. Soc. Polon. Math., 16 (1937), 84--96 (pages 307--318 of the Collected Papers).

\bibitem{Mi} {\sc P. W. Millar}, {\em Martingale integrals}, Trans. AMS, v.~133 (1968), 145--166.

\bibitem{NaVaVo} {\sc F. Nazarov, V. Vasyunin, A. Volberg}, {\em  On Bellman function associated with Chang--Wilson--Wolff theorem}, preprint 2007--2017.

\bibitem{No} {\sc A. A. Novikov}, {\em On stopping times for Wiener processes}, Theory Probab. Appl. 16 (1971), 449-456.

\bibitem{Os2009} {\sc A. Osekowski}, {\em On the best constant in the weak type inequality for the square function of a conditionally symmetric martingale}, Statist. Probab. Lett. Vol. 79, 2009, no. 13, 1536--1538.

\bibitem{OSE} {\sc A.~Osekowski}, {\em Sharp Martingale and Semimartingale Inequalities}, Monografie Matematyczne Vo. 72, Springer, Basel, 2012.


\bibitem{OsW} {\sc A. Osekowski}, {\em Weighted square function inequalities}, Preprint,  pp. 1--15.



\bibitem{Shepp} {\sc L. A. Shepp}, {\em A first passage problem for the Wiener process}, Ann. Math. Statist. {\bf 38} (1967), 1912--1914.


\bibitem{Sza} {\sc S. J. Szarek},  {\em On the best constants in the Khintchine inequality}, Studia Math. 18 (1976),
197--208.

\bibitem{Wang} {\sc G.~Wang}, {\em Sharp square function inequalities for conditionally symmetric martingales}, Transaction of the American Mathematical Society, Vol. 328, No. 1 (1991), 393--419.

\bibitem{Wang2} {\sc G.~Wang}, {\em Some sharp inequalities for conditionally symmetric martingales}, Ph.D. Thesis, University of Illinois at Urbana-Champaign, 1989.

\end{thebibliography}
\end{document}